\documentclass[12pt,reqno]{amsart}
\usepackage{amsmath,amssymb,amsthm,amscd,amsfonts,geometry,color}
\newtheorem{theorem}{Theorem}[section]
\newtheorem{corollary}[theorem]{Corollary}
\newtheorem{lemma}[theorem]{Lemma}
\newtheorem{proposition}[theorem]{Proposition}

\newtheorem{main}{Main Theorem}

\theoremstyle{definition}

\newcommand{\fin}{\operatorname{Fin}}             % finite sets
\newcommand{\doubl}{\operatorname{D}}             % doubletons

\newcommand{\conti}{\operatorname{C}}             % continuous maps
\newcommand{\pseudo}{\operatorname{PM}}             % pseudometrics
\newcommand{\metr}{\operatorname{M}}             % metrics
\newcommand{\adm}{\operatorname{AM}}             % admissible metrics
\newcommand{\sphm}{\operatorname{SPM}}             % pseudometrics of norm 1

\begin{document}

\title[Isometries between the unit spheres of spaces of metrics]{Isometries between the unit spheres of spaces of metrics}
\author{Katsuhisa Koshino}
\address[Katsuhisa Koshino]{Faculty of Engineering, Kanagawa University, 3-27-1 Rokkakubashi, Kanagawa-ku, Yokohama-shi, 221-8686, Japan}
\email{ft160229no@kanagawa-u.ac.jp}
\maketitle

\begin{abstract}
Given a topological space $Z$, let $\pseudo(Z)$ be the space of bounded continuous pseudometrics on $Z$,
 which is endowed with the sup-norm,
 and let $\sphm(Z)$ be the unit sphere of $\pseudo(Z)$.
In this paper, we shall prove that for all non-degenerate compact metrizable spaces $X$ and $Y$, and for any surjective isometry $T : \sphm(X) \to \sphm(Y)$, there exists a homeomorphism $\phi : Y \to X$ such that for any metric $d \in \sphm(X)$ and for any pair of points $(x,y) \in Y^2$, $T(d)(x,y) = d(\phi(x),\phi(y))$.
As a corollary, we can solve a variant of Tingley's problem on spaces of metics.
\end{abstract}

\section{Introduction}

Tingley's problem \cite{Ti}, which was conjectured in 1987,
 is one of the most significant questions in functional analysis.
This asks whether a surjective isometry between the unit spheres of two Banach spaces can be extended to a real-linear surjective isometry over the whole spaces, refer to \cite{Pe} and \cite[Introduction]{Hat} on the recent developments.
T.~Banakh \cite{Banak4} solved the problem for 2-dimensional Banach spaces affirmatively.
Although this question remains open in general Banach spaces,
 it has been solved for specific function spaces, for instance, spaces of continuous functions on topological spaces,
 which traces back to the result due to R.~Wang \cite{Wa}.
In this paper, we shall solve a variant of Tingley's problem on function spaces consisting of metrics on metrizable spaces.
These spaces do not admit the standard algebraic techniques applied in the ones of continuous functions.
Indeed, they lack the unit element and are not closed under multiplication by Urysohn's maps,
 which frequently play key roles in solving Tingley's Problems.
Furthermore, the multivariate nature and the symmetry of metrics prevents the simple peaking structures like the case of univariate functions.

The Banach-Stone theorem \cite{Banac,MHSto} is one of the most fundamental results in the study of isometries between function spaces (cf.~\cite{FJ} as a historical note).
Throughout the paper, an isometry means a surjective isometry.
Given a topological space $Z$, assume that $\conti(Z) = (\conti(Z),\|\cdot\|)$ is the Banach space consisting of real-valued bounded continuous functions on $Z$ with the sup-norm.
Let $\pseudo(Z) \subset \conti(Z^2)$ be the subspace of pseudometrics on $Z$,
 and let $\adm(Z) \subset \pseudo(Z)$ be the subspace of admissible metrics, that is, metrics compatible with the original topology on $Z$.
Moreover, put $\metr(Z)$ be the subspaces ${\rm Pc}(Z)$ or ${\rm Pp}(Z)$ of admissible metrics defined as follows:
\begin{itemize}
 \item For every $d \in {\rm Pc}(Z)$, there are a doubleton $\{z,w\} \subset Z$ and a compact set $K \subset Z$ such that $d(z,w) = \|d\|$ and if $d(x,y) = \|d\|$,
 then $x, y \in K$.
 \item For each $d \in {\rm Pp}(Z)$, there only exists $\{z,w\} \subset Z$ such that $d(z,w) = \|d\|$.
\end{itemize}
If $Z$ is compact,
 then $\adm(Z) = {\rm Pc}(Z)$.
The author \cite{Kos29} (cf.~\cite{Sh}) showed a Banach-Stone type theorem on spaces of metrics as follows:

\begin{theorem}
Let $X$ and $Y$ be metrizable spaces.
The following are equivalent:
\begin{enumerate}
 \item $X$ and $Y$ are homeomorphic;
 \item there exists an isometry $T : \pseudo(X) \to \pseudo(Y)$ with $T(\metr(X)) = \metr(Y)$;
 \item there exists an isometry $T : \adm(X) \to \adm(Y)$ with $T(\metr(X)) = \metr(Y)$;
 \item there exists an isometry $T : \metr(X) \to \metr(Y)$.
\end{enumerate}
In this case, for each isometry $T : \pseudo(X) \to \pseudo(Y)$ with $T(\metr(X)) = \metr(Y)$, there is a homeomorphism $\phi : Y \to X$ such that for all metrics $d \in \pseudo(X)$ and for all points $x, y \in Y$,
 $$T(d)(x,y) = d(\phi(x),\phi(y)).$$
Except for the case that the cardinality of $X$ or $Y$ is equal to $2$, the homeomorphism $\phi$ can be chosen uniquely.
\end{theorem}

Denote by $\sphm(Z)$ the unit sphere of $\pseudo(Z)$.
We will prove the following theorem on isometries between the unit spheres of function spaces consisting of metrics. 

\begin{main}\label{isometry}
Let $X$ and $Y$ be non-degenerate compact metrizable spaces,
 and let $T : \sphm(X) \to \sphm(Y)$ be an isometry.
Then there exists a homeomorphism $\phi : Y \to X$ such that for any $d \in \sphm(X)$ and for any $x, y \in Y$,
\begin{equation}
 T(d)(x,y) = d(\phi(x),\phi(y)). \tag{$\ast$}
\end{equation}
Except for the case where the cardinality of $X$ or $Y$ is equal to $2$, the homeomorphism $\phi$ can be chosen uniquely.
\end{main}

As a corollary, we can provide a positive answer to a variant of Tingley's problem for spaces of metics.

\begin{corollary}\label{Tingley}
Suppose that $X$ and $Y$ are non-degenerate compact metrizable spaces,
 and that $T : \sphm(X) \to \sphm(Y)$ or $T : \sphm(X) \cap \adm(X) \to \sphm(Y) \cap \adm(Y)$ is an isometry.
Then there uniquely exists an extension $\overline{T} : \pseudo(X) \to \pseudo(Y)$ of $T$,
 which is a real-linear\footnote{More precisely, for all pseudometrics $d, \rho \in \pseudo(X)$ and for all numbers $\alpha, \beta \geq 0$, $\overline{T}(\alpha d + \beta \rho) = \alpha\overline{T}(d) + \beta\overline{T}(\rho)$.} isometry.
\end{corollary}

\section{Preliminaries}

In this section, we introduce some basic properties of spaces of metrics,
 which will be used later.
The readers can refer to \cite{Kos20,Kos26,IK2} for the topological and metric structures of spaces consisting of metrics.
On isometries between spaces of metrics, the author \cite{Kos30} proved a variant of the Banach-Stone theorem on spaces consisting of uniformly continuous and lipschitz continuous metrics besides \cite{Kos29}.
For a metrizable space $Z$, $\adm(Z)$ is dense in $\pseudo(Z)$ (cf.~\cite[Proposition~5]{Kos20}).
Similarly, we have the following:

\begin{proposition}\label{dense}
For every metrizable space $Z$, $\sphm(Z) \cap \adm(Z)$ is dense in $\sphm(Z)$.
\end{proposition}

\begin{proof}
For each pseudometric $d \in \sphm(Z)$ and each positive number $\epsilon \in (0,1]$, we will find an admissible metric $\rho \in \sphm(Z) \cap \adm(Z)$ with $\|d - \rho\| \leq \epsilon$.
Fix any admissible metric $d_Z \in \sphm(Z)$ so that $d_Z(x,y) \leq \epsilon$ for all $x, y \in Z$.
Define an admissible metric $\rho \in \sphm(Z) \cap \adm(Z)$ by
 $$\rho(x,y) = \max\{d(x,y),d_X(x,y)\}$$
 for each pair $(x,y) \in Z^2$,
 which is the desired.
\end{proof}

We can obtain the pseudometric version of the Hausdorff metric extension theorem \cite{Hau} as follows (cf.~\cite{Ish8}):

\begin{theorem}\label{ext.}
Let $A$ be a closed subset of a metrizable space $Z$.
For every $d \in \pseudo(A)$, there exists an extension $\tilde{d} \in \pseudo(Z)$ of $d$ with$\|\tilde{d}\| = \|d\|$.
\end{theorem}

\section{Maximal sets and zero sets}

Maximal sets and zero sets are our main tools to prove Main Theorem,
 which are seen in \cite{LNW,ELM}.
For a metrizable space $Z$, denote by $\fin_2(Z)$ the hyperspace consisting of non-empty finite sets in $Z$ of cardinality $\leq 2$,
 which is equipped with the Vietoris topology.
Let $\doubl(Z) \subset \fin_2(Z)$ be the subspace consisting of doubletons in $Z$,
 and for each $\{x,y\} \in \doubl(Z)$, set
 $$\mathcal{P}(Z,\{x,y\}) = \{d \in \sphm(Z) \mid d(x,y) = \|d\| = 1\}.$$
Given a pseudometric $d \in \sphm(Z)$, define the maximal set and the zero set of $d$ as follows:
 $$\mathcal{M}(Z,d) = \{\{x,y\} \in \doubl(Z) \mid d(x,y) = \|d\| = 1\},$$
 $$\mathcal{Z}(Z,d) = \{\{x,y\} \in \doubl(Z) \mid d(x,y) = 0\}.$$

\begin{lemma}\label{decomp.}
Let $Z$ be a compact metrizable space.
For each $d \in \sphm(Z)$ and each $\{u,v\} \in \doubl(Z)$ with $0 < d(u,v) < \|d\| = 1$, there exists $\rho \in \sphm(Z)$ such that $\rho(u,v) = 0$, $\mathcal{M}(Z,d) \cap \mathcal{Z}(Z,\rho) = \emptyset$, and $\mathcal{M}(Z,\rho) \cap \mathcal{Z}(Z,d) = \emptyset$.
\end{lemma}

\begin{proof}
For every point $z \in Z$, let
 $$Z_z = \{x \in Z \mid d(x,z) = 0\}.$$
Since $Z$ is compact and $d$ is a pseudometric on $Z$,
 we can decompose $Z$ by the following mutually disjoint compact subsets:
 $$\mathcal{D} = \{Z_z \mid z \in Z \text{ such that } d(z,u), d(z,v) > 0\} \cup \{Z_u \cup Z_v\}.$$
This induces the natural quotient space $Z/\mathcal{D}$ and the natural quotient map
 $$q : Z \ni z \to Z_z \in Z/\mathcal{D},$$
 where if $z \in Z_u \cup Z_v$,
 then $q(z) = Z_u \cup Z_v$.
To prove that $q$ is a closed map, fix any closed subset $A \subset Z$.
For every convergent sequence $\{x_n\} \subset q^{-1}(q(A))$ and each positive integer $n$, there is a point $y_n \in A$ such that $q(x_n) = q(y_n)$.
Since $A$ is compact,
 we can replace $\{y_n\} \subset A$ with a convergent subsequence,
 and then $\lim_{n \to \infty} y_n \in A$.
Observe that the relation
 $$R = \{(x,y) \in Z^2 \mid q(x) = q(y)\}$$
 is closed in $Z^2$.
Indeed, for all pairs $(x,y) \in Z^2 \setminus R$, since $q(x) \neq q(y)$,
 $d(x,y) > 0$,
 and moreover we may assume that $d(x,u), d(x,v) > 0$ without loss of generality.
It follows from the continuity of $d$ that there exist open neighborhoods $U$ and $V$ of $x$ and $y$ in $Z$ respectively, such that for every $(z,w) \in U \times V$, $d(z,w) > 0$ and $d(z,u), d(z,v) > 0$.
Then $q(z) \neq q(w)$,
 so $(z,w) \in Z^2 \setminus R$,
 which means that $R$ is a closed subset of $Z^2$.
Therefore the sequence $\{(x_n,y_n)\}$ is converging in $R$,
 and hence
 $$q\Big(\lim_{n \to \infty} x_n\Big) = q\Big(\lim_{n \to \infty} y_n\Big) \in q(A).$$
Then $\lim_{n \to \infty} x_n \in q^{-1}(q(A))$,
 which implies that $q^{-1}(q(A))$ is closed.
It follows that $q$ is a closed map.

By virtue of \cite[Proposition~2.4.9 and Theorem~4.2.13]{E}, the quotient space $Z/\mathcal{D}$ is also metrizable,
 so fix any admissible metric $d' \in \adm(Z/\mathcal{D})$ with $\|d'\| = 1$.
Since the quotient map $q$ is a continuous surjection,
 due to \cite[Lemma~2.4]{IK2}, $\rho : Z^2 \to [0,\infty)$ defined by
 $$\rho(x,y) = d'(q(x),q(y))$$
 is a pseudometric of the sup-norm $1$,
 that is the desired.
Indeed,
 $$\rho(u,v) = d'(q(u),q(v)) = 0$$
 because $q(u) = q(v)$.
For any doubleton $\{x,y\} \in \mathcal{M}(Z,d)$, since
 $$d(x,y) = \|d\| = 1 > d(u,v) > 0,$$
 $Z_x \cap Z_y = \emptyset$ and $\{x,y\} \not\subset Z_u \cup Z_v$.
Therefore $q(x) \neq q(y)$,
 and hence
 $$\rho(x,y) = d'(q(x),q(y)) > 0,$$
 which means that $\mathcal{M}(Z,d) \cap \mathcal{Z}(Z,\rho) = \emptyset$.
For any doubleton $\{x,y\} \in \mathcal{Z}(Z,d)$, since $d(x,y) = 0$,
 $q(x) = q(y)$.
Thus
 $$\rho(x,y) = d'(q(x),q(y)) = 0.$$
It follows that $\mathcal{M}(Z,\rho) \cap \mathcal{Z}(Z,d) = \emptyset$.
We complete the proof.
\end{proof}

From now on, we assume that $X$ and $Y$ are non-degenerate metrizable spaces and $T : \sphm(X) \to \sphm(Y)$ is an isometry.

\begin{lemma}\label{intersect.}
If $Y$ is compact,
 then
 $$\bigcap_{d \in \mathcal{P}(X,\{x,y\})} \mathcal{M}(Y,T(d)) \neq \emptyset$$
 for every doubleton $\{x,y\} \in \doubl(X)$.
\end{lemma}

\begin{proof}
Since $Y$ is compact and the natural surjection
 $$i : Y^2 \ni (z,w) \mapsto \{z,w\} \in \fin_2(Y)$$
 is continuous (cf.~\cite[Lemma~5.3.4]{vM2}),
 the hyperspace $\fin_2(Y)$ is also compact.
For every $d \in \mathcal{P}(X,\{x,y\})$, it follows from the continuity of $T(d)$ that $\mathcal{M}(Y,T(d))$ is a closed set in $\fin_2(Y)$, see the proof of Lemma~3.6 of \cite{Kos29},
 and hence $\mathcal{M}(Y,T(d))$ is compact.

We will show the finite intersection property of the family
 $$\{\mathcal{M}(Y,T(d)) \mid d \in \mathcal{P}(X,\{x,y\})\}.$$
Fix any positive integer $n$ and any $d_i \in \mathcal{P}(X,\{x,y\})$, $1 \leq i \leq n$.
Letting $d = \sum_{i = 1}^n d_i/n$,
 we check that for each $i \in \{1, \ldots, n\}$,
 $$\{\rho \in \sphm(X) \mid \|\rho - d_i\| < 1\} \subset \{\rho \in \sphm(X) \mid \|\rho - d\| < 1\}.$$
Indeed, for any $\rho \in \sphm(X)$ with $\|\rho - d_i\| < 1$,
 $$\|\rho - d\| = \Bigg\|\rho - \sum_{j = 1}^n \frac{d_j}{n}\Bigg\| \leq \frac{1}{n}\sum_{j = 1}^n \|\rho - d_j\| = \frac{1}{n}\sum_{j \neq i} \|\rho - d_j\| + \frac{\|\rho - d_i\|}{n} < 1.$$
Since $T$ is isometric,
 we also have that
 $$\{\rho \in \sphm(Y) \mid \|\rho - T(d_i)\| < 1\} \subset \{\rho \in \sphm(Y) \mid \|\rho - T(d)\| < 1\}.$$
Next, we will verify that $\mathcal{M}(Y,T(d)) \subset \mathcal{M}(Y,T(d_i))$.
Assume that there is $\{u,v\} \in \mathcal{M}(Y,T(d)) \setminus \mathcal{M}(Y,T(d_i))$.
Then $T(d_i)(u,v) < 1 = T(d)(u,v)$.
Supposing that $T(d_i)(u,v) = 0$,
 we get that
 $$1 = T(d)(u,v) - T(d_i)(u,v) \leq \|T(d) - T(d_i)\| < 1,$$
 which is a contradiction.
Hence $0 < T(d_i)(u,v) < 1$.
Applying Lemma~\ref{decomp.}, we can find a pseudometric $\rho \in \sphm(Y)$ such that $\rho(u,v) = 0$, $\mathcal{M}(Y,T(d_i)) \cap \mathcal{Z}(Y,\rho) = \emptyset$, and $\mathcal{M}(Y,\rho) \cap \mathcal{Z}(Y,T(d_i)) = \emptyset$.
Then $\|\rho - T(d_i)\| < 1$,
 but
 $$\|\rho - T(d)\| \geq |\rho(u,v) - T(d)(u,v)| = 1.$$
This yields a contradiction.
Consequently, $\mathcal{M}(Y,T(d)) \subset \mathcal{M}(Y,T(d_i))$.
According to \cite[Theorem~3.1.1]{E},
 $$\bigcap_{d \in \mathcal{P}(X,\{x,y\})} \mathcal{M}(Y,T(d)) \neq \emptyset.$$
The proof is completed.
\end{proof}

Combining the above lemma with the same argument as Proposition~3.8 of \cite{Kos29}, we can obtain the following bijection between $\doubl(X)$ and $\doubl(Y)$.

\begin{proposition}\label{bij.}
Assume that $X$ and $Y$ are compact.
There is a bijection $\Phi : \doubl(Y) \to \doubl(X)$ such that
 $$T(\mathcal{P}(X,\{x,y\})) = \mathcal{P}(Y,\Phi^{-1}(\{x,y\}))$$
 for any doubleton $\{x,y\} \in \doubl(X)$ and
 $$T^{-1}(\mathcal{P}(Y,\{z,w\})) = \mathcal{P}(X,\Phi(\{z,w\}))$$
 for any doubleton $\{z,w\} \in \doubl(Y)$.
\end{proposition}

\begin{proof}
In fact, we can define a bijection $\Phi : \doubl(Y) \to \doubl(X)$ as follows:
 $$\{\Phi(\{z,w\})\} = \bigcap_{\rho \in \mathcal{P}(Y,\{z,w\})} \mathcal{M}(X,T^{-1}(\rho)) \text{ and } \{\Phi^{-1}(\{x,y\})\} = \bigcap_{d \in \mathcal{P}(X,\{x,y\})} \mathcal{M}(Y,T(d))$$
 for each $\{z,w\} \in \doubl(Y)$ and each $\{x,y\} \in \doubl(X)$.
\end{proof}

\section{Composition operators by homeomorphisms}

This section is devoted to proving that $T$ is a composition operator by certain homeomorphism form $Y$ to $X$.

\begin{lemma}\label{transl.}
Let $Z$ be a metrizable space.
For every $d \in \sphm(Z)$ and every $\{x,y\} \in \doubl(Z)$, there exists $\rho \in \conti(Z^2)$ such that $d + \rho \in \mathcal{P}(Z,\{x,y\})$ and $\|\rho\| = 1 - d(x,y)$.
\end{lemma}

\begin{proof}
In the case that $d(x,y) = 1$, we may choose $\rho = 0 \in \conti(Z^2)$.
In the case that $d(x,y) = 0$, fix any metric $\rho_0 \in \mathcal{P}(Z,\{x,y\})$.
Let $\rho = \rho_0 - d \in \conti(Z^2)$,
 so
 $$1 \geq \|\rho\| = \|\rho_0 - d\| \geq \rho_0(x,y) - d(x,y) = 1 - d(x,y) = 1,$$
 that is, $\|\rho\| = 1 - d(x,y)$.

In the case that $0 < d(x,y) < 1$, put
 $$\rho_0 = \min\bigg\{\frac{d}{d(x,y)},1\bigg\}.$$
Observe that $\rho_0 \in \mathcal{P}(Z,\{x,y\})$.
Define $\rho = \rho_0 - d \in \conti(Z^2)$.
To show that it is the desired function, take any $z, w \in Z$.
When $d(z,w) > d(x,y)$,
 $$\rho(z,w) = \rho_0(z,w) - d(z,w) = \min\bigg\{\frac{d(z,w)}{d(x,y)},1\bigg\} - d(z,w) = 1 - d(z,w) < 1 - d(x,y).$$
When $d(z,w) = d(x,y)$,
 $$\rho(z,w) = \rho_0(z,w) - d(z,w) = \min\bigg\{\frac{d(z,w)}{d(x,y)},1\bigg\} - d(z,w) = 1 - d(x,y).$$
When $d(z,w) < d(x,y)$,
 $$\rho(z,w) = \rho_0(z,w) - d(z,w) = \min\bigg\{\frac{d(z,w)}{d(x,y)},1\bigg\} - d(z,w) = \frac{d(z,w)(1 - d(x,y))}{d(x,y)} < 1 - d(x,y).$$
It follows that $\|\rho\| = 1 - d(x,y)$.
The proof is completed.
\end{proof}

From now on, assume that $X$ and $Y$ are compact metrizable spaces,
 and let $\Phi : \doubl(Y) \to \doubl(X)$ be a bijection as in Proposition~\ref{bij.}.
For simplicity, we shall write $d(\Phi\{z,w\}) = d(x,y)$ for each $d \in \pseudo(X)$ and each $\{z,w\} \in \doubl(Y)$ with $\Phi(\{z,w\}) = \{x,y\} \in \doubl(X)$.

\begin{proposition}\label{eq.}
For any pseudometric $d \in \pseudo(X)$ and any doubleton $\{x,y\} \in \doubl(Y)$,
 $$T(d)(x,y) = d(\Phi(\{x,y\}))$$
 holds.
\end{proposition}

\begin{proof}
In the case that $T(d)(x,y) = 1$, the equality follows from Proposition~\ref{bij.} immediately.
When $T(d)(x,y) < 1$,
 we prove that $T(d)(x,y) \leq d(\Phi(\{x,y\}))$.
Using Lemma~\ref{transl.}, we can take $\rho \in \conti(Y^2)$ so that $T(d) + \rho \in \mathcal{P}(Y,\{x,y\})$ and $\|\rho\| = 1 - T(d)(x,y)$.
By Proposition~\ref{bij.},
 $$T^{-1}(\mathcal{P}(Y,\{x,y\})) = \mathcal{P}(X,\Phi(\{x,y\})),$$
 and hence $T^{-1}(T(d) + \rho) \in \mathcal{P}(X,\Phi(\{x,y\}))$.
Then
 $$T^{-1}(T(d) + \rho)(\Phi(\{x,y\})) = 1.$$
Since $T$ is isometry,
 \begin{align*}
  1 - d(\Phi(\{x,y\})) &= T^{-1}(T(d) + \rho)(\Phi(\{x,y\})) - d(\Phi(\{x,y\}))\\
  &\leq \|T^{-1}(T(d) + \rho) - d\| = \|T(d) + \rho - T(d)\| = \|\rho\| = 1 - T(d)(x,y),
 \end{align*}
 which implies that $T(d)(x,y) \leq d(\Phi(\{x,y\}))$.
Similarly, we get that $T(d)(x,y) \geq d(\Phi(\{x,y\}))$.
The proof is completed.
\end{proof}

We will construct a homeomorphism from $Y$ to $X$ that is compatible with the bijection $\Phi$.

\begin{proposition}\label{corresp.}
There exists a homeomorphism $\phi : Y \to X$ such that $\Phi(\{x,y\}) = \{\phi(x),\phi(y)\}$ for all points $x, y \in Y$ and $\Phi^{-1}(\{z,w\}) = \{\phi^{-1}(z),\phi^{-1}(w)\}$ for all points $z, w \in X$.
\end{proposition}

\begin{proof}
Recall that the cardinality of $X$ is coincident with the one of $Y$ because $\Phi$ is a bijection.
The case where the cardinality of $X$ and $Y$ is equal to $2$ is valid obviously.
In the case where the cardinality of $X$ and $Y$ is greater than $2$, take any pair $\{z_1, z_2\} \subset Y \setminus \{y\}$ of distinct points.
Then the intersection $\Phi(\{y,z_1\}) \cap \Phi(\{y,z_2\})$ is a singleton.
Suppose not,
 so since $\Phi$ is an injection,
 $\Phi(\{y,z_1\}) \cap \Phi(\{y,z_2\}) = \emptyset$.
By virtue of Theorem~\ref{ext.}, we can choose $d \in \sphm(X)$ so that $d(\Phi(\{y,z_1\})) = d(\Phi(\{y,z_2\})) = 1/3$ and $d(\Phi(\{z_1,z_2\})) = 1$.
Due to Proposition~\ref{eq.},
\begin{align*}
 1 &= d(\Phi(\{z_1,z_2\})) = T(d)(z_1,z_2) \leq T(d)(y,z_1) + T(d)(y,z_2)\\
 &= d(\Phi(\{y,z_1\})) + d(\Phi(\{y,z_2\})) = 2/3,
\end{align*}
 which is a contradiction.
Thus the doubletons $\Phi(\{y,z_1\})$ and $\Phi(\{y,z_2\})$ intersect at the only point of $X$.
Moreover, by the similar argument, $\bigcap_{z \in Y \setminus \{y\}} \Phi(\{y,z\})$ is also a singleton, see the proof of \cite[Lemma~4.2]{Kos29}.
Hence we can define a map $\phi : Y \to X$ by
 $$\{\phi(y)\} = \bigcap_{z \in Y \setminus \{y\}} \Phi(\{y,z\}).$$
Furthermore, the inverse $\phi^{-1}$ of $\phi$ can be found as follows:
 $$\{\phi^{-1}(x)\} = \bigcap_{z \in X \setminus \{x\}} \Phi^{-1}(\{x,z\})$$
 for every point $x \in X$, refer to the proof of \cite[Proposition~4.3]{Kos29}.

The compatibility of $\Phi$ and $\phi$ follows from the definition immediately.
Using the continuity of pseudometrics in $\sphm(X)$ and $\sphm(Y)$, we can prove that $\phi$ and $\phi^{-1}$ are continuous, refer to the proof of \cite[Main Theorem]{Kos29}.
Consequently, $\phi : Y \to X$ is a homeomorphism.
\end{proof}

\section{Proof of Main Theorem}

Now we are in a position to show Main Theorem and Corollary~\ref{Tingley}.

\begin{proof}[Proof of Main Theorem]
The formula ($\ast$) follows from Propositions~\ref{eq.} and \ref{corresp.} directly.
In the case that the cardinality of $X$ and $Y$ is greater than $2$,
 we can show the uniqueness of $\phi$ easily, see the proof of Main Theorem in \cite{Kos29}.
We complete the proof.
\end{proof}

\begin{proof}[Proof of Corollary~\ref{Tingley}]
Remark that for any metrizable space $Z$, $\sphm(Z)$ is a complete metric space and $\sphm(Z) \cap \adm(Z)$ is dense in $\sphm(Z)$ by Proposition~\ref{dense}.
Hence each isometry $T : \sphm(X) \cap \adm(X) \to \sphm(Y) \cap \adm(Y)$ can be extended to an isometry between $\sphm(X)$ and $\sphm(Y)$.
By virtue of Main Theorem, for every isometry $T : \sphm(X) \to \sphm(Y)$, there exists a homeomorphism $\phi : Y \to X$ such that $T(d)(x,y) = d(\phi(x),\phi(y))$ for all pseudometrics $d \in \sphm(X)$ and all points $x, y \in Y$.
Then using the homeomorphism $\phi$, we can define $\overline{T} : \pseudo(X) \to \pseudo(Y)$ by $\overline{T}(d)(x,y) = d(\phi(x),\phi(y))$.
It is obvious that $\overline{T}$ is an extension of $T$,
 which is real-linear.
The uniqueness of $\overline{T}$ follows form the linearity of $\overline{T}$.
Moreover, by the same argument as the proof of the implication (1) $\Rightarrow$ (2) in \cite[Main Theorem]{Kos29}, $\overline{T}$ is an isometry.
The proof is completed.
\end{proof}

\end{document}